\documentclass[12pt,]{amsart}
\usepackage[foot]{amsaddr}
\usepackage{amssymb,amsmath,latexsym,amscd,amsthm,mathtools}
\usepackage{graphics}
\usepackage{setspace}
\usepackage{enumerate}
\usepackage{float}
\usepackage[size=footnotesize]{caption}

\usepackage[
backend=bibtex,
style=numeric,
maxcitenames=1,
]{biblatex}
\usepackage{xcolor} 
\usepackage{bbm,bm}
\usepackage{float}
\usepackage{hyperref, fullpage}
	\hypersetup{colorlinks=true, linkcolor=blue, filecolor=magenta, urlcolor=cyan, pdftitle={Overleaf Example}, pdfpagemode=FullScreen}  
\usepackage{verbatim}
\usepackage[normalem]{ulem}

\newcommand{\card}[1]{\lvert #1\rvert} 
\newcommand{\ZZ}{{\mathbb Z}}

\newtheorem{thm}{Theorem}[section]

\newtheorem{lem}[thm]{Lemma}

\newtheorem{conj}[thm]{Conjecture}

\theoremstyle{definition}

\begin{document}
\title{Bootstrap Percolation, Connectivity, and Graph Distance}
\author{Hudson LaFayette}
\email{lafayettehl@vcu.edu}

\author{Rayan Ibrahim}
\email{ibrahimr3@vcu.edu}
\urladdr{https://raymaths.github.io}

\author{Kevin McCall$^{\ast}$}
\address{$^{\ast}$Corresponding author}
\email{mccallkj@alumni.vcu.edu}

\address{Department of Mathematics \& Applied Mathematics \\ 
   Virginia Commonwealth University \\Richmond, VA 23284 \\ USA  
   }



\begin{abstract}
    Bootstrap Percolation is a process defined on a graph which begins with an initial set of infected vertices.  In each subsequent round, an uninfected vertex becomes infected if it is adjacent to at least $r$ previously infected vertices.  If an initially infected set of vertices, $A_0$, begins a process in which every vertex of the graph eventually becomes infected, then we say that $A_0$ percolates.  In this paper we investigate bootstrap percolation as it relates to graph distance and connectivity.  We find a sufficient condition for the existence of cardinality 2 percolating sets in diameter 2 graphs when $r = 2$.  We also investigate connections between connectivity and bootstrap percolation and lower and upper bounds on the number of rounds to percolation in terms of invariants related to graph distance.
\end{abstract}

\maketitle

\section{Introduction}

Bootstrap percolation is a process defined on a graph, $G$.  The process begins with an initial set of infected vertices $A_0 \subseteq V(G)$. In each subsequent round, an uninfected vertex, $v$, becomes infected if $v$ is adjacent to at least $r$ previously infected vertices.  Once infected, vertices remain infected.  We use $A_t$ to denote the set of all infected vertices as of round $t$.  Symbolically, \[A_t = A_{t-1} \cup \{v \in V(G): |N(v) \cap A_{t-1}| \geq r\}\]

The parameter $r$ is called the percolation threshold.  If $G$ is a finite graph, then after a finite number of rounds, either all vertices of $G$ become infected or the infection stops at some proper subset of $V(G)$.  The set of infected vertices after the percolation process finishes is called the closure of $A_0$, denoted $\langle A_0 \rangle$. If $\langle A_0 \rangle = V(G)$, then we say that $A_0$ is contagious or $A_0$ percolates.

Bootstrap percolation was introduced by \textcite{Chalupa_1979}. One model that has received much attention is when the vertices of $A_0$ are selected randomly; each vertex is selected independently and every vertex of $G$ has probability $p$ of being initially selected. After the initial step, the infection proceeds deterministically.  This model has been studied extensively, for example in \cite{Balogh_2006a, Balogh_2012, Balogh_2006b, Balogh_1998, Balogh_2007, Holroyd_2003}. 

Another area of study is extremal problems.  The minimum size of a percolating set in a graph $G$ with percolation threshold $r$ is denoted $m(G,r)$.  Observe that if $|V(G)|$ is at least $r$, then $m(G,r) \geq r$. \textcite{Freund_2018} showed that for a graph $G$ of order $n$, if $\delta(G) \geq \frac{r-1}{r}n $ then $m(G,r) = r$.  Let $\sigma_2(G)$ be the minimum sum of degrees over all pairs of non-adjacent vertices of $G$.  \textcite{Freund_2018} proved that if $G$ satisfies Ore's condition, i.e., $\sigma_2(G) \geq n$, then  $m(G,2) = 2$.  Furthermore, they proved that both of these bounds are sharp.

\textcite{Gunderson_2020} extended the first result by showing that if the order of $G$ is sufficiently large, then the bound on the minimum degree can be weakened.  \textcite{Wesolek_2019} extended Gunderson's result by proving a lower bound on the minimum degree sufficient to guarantee a percolating set of size $\ell \geq r$.  \textcite{Dairyko_2020} extended \textcite{Freund_2018}'s theorem on Ore's condition by characterizing the graphs for which $\sigma_2(G) \geq n$ and $\sigma_2(G) \geq n - 1$ is required to guarantee $m(G,2) = 2$.  For all other graphs, $\sigma_2(G) \geq n - 2$ is sufficient.  Degree conditions on bootstrap percolation have also been studied in \textcite{Reichman_2012}. \textcite{Bushaw_2022} investigated other conditions for which $m(G,2) = 2$.  

Another problem is investigating $m(G,r)$ for particular classes of graphs.  One class which has received significant attention is the $d$-dimensional lattice on $n^d$ vertices, denoted $[n]^d$.  This has been studied in \cite{ Balogh_2010, Balogh_1998, Hambardzumyan_2020, Morrison_2018, Przykucki_2020}.

In this paper, we continue the investigations of \textcite{Bushaw_2022} by studying bootstrap percolation from the perspective other non-degree conditions.  In particular, we focus on diameter.  We begin with a proof of a conjecture from \cite{Bushaw_2022}.  Suppose $G$ is a connected graph of order at least 3 with at most 2 blocks.  If $G$ is diameter 2 and contains no induced cycle of length 5, then $m(G,2) = 2$.  In section 3, we explore the consequences of percolating sets of size $r$ on the connectivity of a graph.  In section 4, we examine the minimum number of rounds to percolation given the size of the percolating set in relation to the diameter and radius of a graph.  In section 5, we investigate the maximum number of rounds to percolation in terms of graph distance.  The problem of the number of rounds to percolation has also been investigated in \cite{Benevides_2015,Przykucki_2012}.  We close with some open problems.  

\section{A sufficient condition for 2-bootstrap percolation}

Before introducing the conjecture, we provide some background and definitions.  If a graph $G$ contains at least one pair vertices which percolate when $r = 2$, then we say that $G$ is 2-bootstrap good or 2-BG. A \textit{block} of $G$ is a maximally 2-connected subgraph of $G$.  If $B$ is a block of $G$, then we use $G[B]$ to denote the subgraph of $G$ induced by $B$.  It is shown in \textcite{Bushaw_2022} that a graph with more than two blocks cannot be 2-BG.  Since disconnected graphs of order more than two also cannot be 2-BG, we only concern ourselves with connected graphs.  Furthermore, since graphs of order less than two are trivially 2-BG, we only examine graphs of order three or more.  Hence, we define the set $\mathcal{G}$ as the collection of all connected graphs of order 3 or more with at most two blocks.  

A graph $G$ has a \textit{dominating vertex} if $G$ contains a vertex $v$ adjacent to all other vertices of $G$.  A graph $G$ is \textit{locally connected} if the open neighborhood of every vertex forms a connected graph.  We present the following lemma:

\begin{lem}\label{lem:dom}
If a graph $G$ is 2-connected and has a dominating vertex, then $G$ is locally connected. 
\end{lem}

\begin{proof}
Let $v$ be a dominating vertex of $G$ and let $u$ be some other vertex of $G$. Since $v$ is in the open neighborhood of $u$, any two vertices in $N(u)$ are joined by $v$.  Hence, $N(u)$ is connected. As $G$ is 2-connected, $v$ cannot be a cut vertex. Hence, $(V(G) \setminus \{v\}) = N(v)$ is also connected.  
\end{proof}

The following are Theorem 2.16 and Conjecture 4.1 respectively in \cite{Bushaw_2022}.
\begin{lem}\label{lem:local2bg}
If a graph $G \in \mathcal{G}$ is locally connected, then it is 2-BG.  Furthermore, if $G$ has no leaf, then any pair of adjacent vertices will percolate in $G$.    
\end{lem}

\begin{conj}
If a graph in $\mathcal{G}$ is perfect and its diameter is no more than 2 then the graph is 2-bootstrap good.  
\end{conj}

We present the following theorem, weakening the assumption that $G$ is perfect. 

\begin{thm}
If a graph $G \in \mathcal{G}$ has diameter 2 and contains no induced cycle of length 5, then $G$ is 2-bootstrap good.
\end{thm}

We divide the proof into two cases.

\noindent \textbf{Case 1}: $G$ has $2$ blocks. 

\begin{proof}
In this case, we do not need the assumption that $G$ contains no induced cycles of length 5 or more. Let $v$ be a cut vertex of $G$, and let $B_1$ and $B_2$ be the blocks of $G$. Since $G$ has diameter $2$, $v$ is dominating in $G$. Hence, $G[B_1]$ and $G[B_2]$ are locally connected by Lemma \ref{lem:dom}. Pick $w\in B_1$ and $x \in B_2$ with $\{w,x\}$ as the initial infected set, which then infect $v$. Then, $\{w,v\}$ percolates in $B_1$ and $\{v,x\}$ percolates in $B_2$ by Lemma \ref{lem:local2bg}. So $G$ is $2$-BG, where any pair of vertices, with one vertex of the pair in $B_1-v$ and the other in $B_2-v$, percolates in $G$.
\end{proof}

\noindent \textbf{Case 2}: $G$ is 2-connected. 

\begin{proof}
Assume that $G$ is 2-connected, has diameter 2, and contains no induced cycles of length 5 or more. Suppose toward a contradiction that $G$ is not 2-BG.  Let $H$ be a maximal 2-connected, 2-BG subgraph of $G$. In other words, any subgraph of $G$ containing $H$ (other than $H$ itself) fails to be 2-connected or fails to be 2-BG. Observe that any vertex in $V(G) - V(H)$ has at most one neighbor in $H$. Since $G$ is connected and $H$ is a proper subgraph of $G$, there is a vertex $v \in V(G) - V(H)$ with exactly one neighbor, $w$, in $H$. \\

\textbf{Claim 1}: $w$ is adjacent to every vertex in $H$.  \\

\textbf{Proof of Claim 1}:  

\indent Suppose towards a contradiction that $w$ is not adjacent to some vertex $z \in V(H)$.  Since $G$ has diameter 2, there is some vertex $y\in V(G)$ such that $y$ is adjacent to both $w$ and $z$.  Observe, since $y$ is adjacent to $w$ and $z$, i.e. $y$ has $2$ neighbors in $H$, it must be that $y \in V(H)$. Since $v$ is only adjacent to a single vertex in $H$ and $G$ is diameter 2, there must be some vertex, $v'$ outside of $H$ such that $v'$ is adjacent to both $v$ and $z$ (see Figure \ref{fig:Claim1}). 

\begin{figure}[h]
    \centering
    \includegraphics[scale=1]{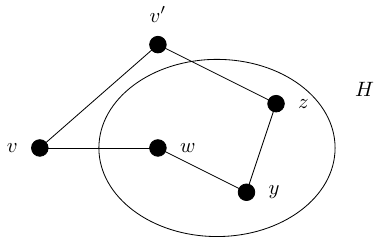}
    \caption{Claim 1.}\label{fig:Claim1}
\end{figure}

\indent Recall that a vertex outside $H$ can only be adjacent to a single vertex within $H$.  Hence, $v'$ cannot be adjacent to any vertex in $H$ other than $z$. The five vertices $v,v',w,y,z$ form an induced 5-cycle, contradicting our assumption that $G$ has no induced 5-cycles.  This proves our claim that $w$ is adjacent to every vertex in $H$.

If $w$ was the only vertex in $H$ adjacent to vertices outside of $H$, then $w$ would be a cut vertex, contradicting the assumption that $G$ is $2$-connected. So there must be a vertex $w'$ in $H$ with a neighbor $v'$ outside of $H$. Note that $v \ne v'$, as $v$ and $v'$ each have a unique neighbor in $H$. We now have two cases.

\textbf{Case 2a}: $v$ is adjacent to $v'$.  

If so, then infect $\{v,w'\}$.  These in turn infect $v'$ along with $w$.  By Lemma \ref{lem:local2bg}, $\{w,w'\}$ infects all of $H$.  But this means that $H \cup \{v,v'\}$ is a 2-connected, 2-BG subgraph  of $G$ containing $H$, in contradiction to our earlier assumptions.

\begin{figure}[h]
    \centering
    \includegraphics[scale=1]{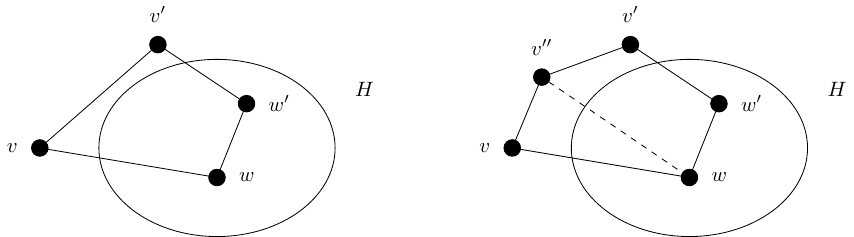}
    \caption{Case 2a on the left and Case 2b on the right.}\label{fig:Claim2}
\end{figure}

\textbf{Case 2b}: $v$ is not adjacent to $v'$. 

Since $G$ has diameter 2, there must be a vertex $v''$ which joins $v$ and $v'$. This vertex cannot be in $H$ because each of $v,v'$ is only adjacent to a single vertex in $H$. 

We now have two possibilities: $v''$ is adjacent exactly one of $w$ or $w'$; or $v''$ is adjacent to neither $w$ nor $w''$.  If $v''$ is adjacent to neither, then $v,v'',v',w',w$ form an induced 5-cycle.  If $v''$ is adjacent to $w$, then we can infect $v'',w'$, which in turn infect $v',w$ and $H' = H \cup \{v',v''\}$ forms a 2-BG, 2-connected subgraph containing $H$.  If $v''$ is adjacent to $w'$, then the situation is similar except that $H' = H \cup \{v,v''\}$.  This is shown in Figure \ref{fig:Claim2}. 

Cases 2a and 2b both lead to contradictions, so we conclude that there can be no such $H$ and $G$ must be 2-BG.
\end{proof}

\section{Connectivity and Bootstrap Percolation}

Let \(r\in\ZZ^+\). Similar to the definition of \(2\)-BG, if a graph \(G\) contains at least one set of \(r\) vertices which percolate, then \(G\) is \(r\)-bootstrap good or \(r\)-BG. A graph \(G\) is \emph{\(k\)-connected} if it has at least \(k+1\) vertices and does not contain a cut set of size \(k-1\).  Recall that a block is a maximal induced 2-connected subgraph of $G$.\par  
A graph is 1-BG if and only if it is connected.  In \cite{Bushaw_2022}, the following result is Lemma 2.1:

\begin{lem}
    \label{lem:2-BG graph structure}
    If a graph is \(2\)-BG, then it has at most two blocks.  
\end{lem}

In this section, we seek to expand on the result in Lemma~\ref{lem:2-BG graph structure} by investigating the effect of percolating sets of size $r$, where $r \geq 3$, on the connectivity of graphs. This topic was investigated independently by \textcite{Flippen_2022}, who showed that \(3\)-BG graphs have at most three leaf blocks (a block that is a leaf in a block-cut graph).  A natural first question is, ``what is the maximum number of blocks of an $r$-BG graph?"  Before answering this question, we present the following lemma: 

\begin{lem}
\label{lem:component structure lemma}
Let $G$ be an $r$-BG graph with at least $r + 1$ vertices and $A_0$ be a cardinality $r$ percolating set of $G$. If $X$ is a cut set of $G$ with $\card{X}<r$ and $K$ is the set of components of $G - X$ which are not contained in $A_0$, then \(\card{V(C)\cap A_0}\geq r-\card{X}\) for each \(C\in K\). Moreover,  $|K| \leq \lfloor r/(r - \card{X}) \rfloor$ and if $|K| \geq 2$, then $r/2 \leq |X| \leq r - 1$. 
\end{lem}

\begin{proof}
The first part of the second statement implies the second part of the second statement since $\card{X} < r/2$ implies $r-\card{X} > r/2$ and thus $|K| \leq r/(r-\card{X}) < 2$. \par

Suppose $C$ is a component of $G - X$ and $C \in K$.  Since $C \in K$, there is some vertex, $v$, in $C$ which is not initially infected.  Without loss of generality, we may let $v$ be the earliest infected vertex of $C$ which is not initially infected (it is possible that there are multiple choices for $v$).  Since $v$ is the earliest infected vertex, $v$ cannot be infected by other vertices of $C$ and in fact can only be infected by vertices of $A_0$ or $X$, i.e., $\card{N(v) \cap (X \cup A_0)} \geq r$. Let \(i:=r- \card{X}\). Since $N(v) \subseteq V(C)\cup X$ and $\card{X} = r - i$ we must have $\card{A_0 \cap V(C)} \geq i$. No two components of $G - X$ have any vertices in common, so $|K| \cdot i \leq |A_0| = r$, which implies that $K \leq r/i$.
\end{proof}

Throughout this section, we will use the notation from the above lemma: $G$ is a graph, $A_0$ is a percolating set of $G$, and $X$ is a cut set of $G$.  For simplicity, we will use the term component to refer to a subgraph of $G$ induced by a component of $G - X$. Observe that a cut set $X$, when $|X| < r$, separates any percolating set of size $r$. If all vertices of $A_0$ are in the same component of $G - X$, then no other component of $G - X$ can become infected. Likewise, no component can have zero vertices of $A_0$, otherwise no vertices of the component would be able to become infected. Since each component must have at least one vertex of $A_0$, we can have at most $r$ components of $G - X$. By Lemma \ref{lem:component structure lemma}, this can only occur when $|X| = r - 1$. In fact, this bound is sharp.  Here is one family of graphs which attains the bound: let $G$ be a graph with $r$ disjoint nonempty complete subgraphs $H_1, H_2, ..., H_r$ and let $X$ be a set of $r - 1$ vertices each adjacent to every vertex in every $H_i$. Then, select one vertex from each $H_i$ to be initially infected. These $r$ vertices infect $X$. Then each $H_i$ is infected by $X$ together with its single infected vertex. See Figure \ref{fig:3components} for an example when \(r=3\).

\begin{figure}[h]
    \centering
    \includegraphics[scale=1]{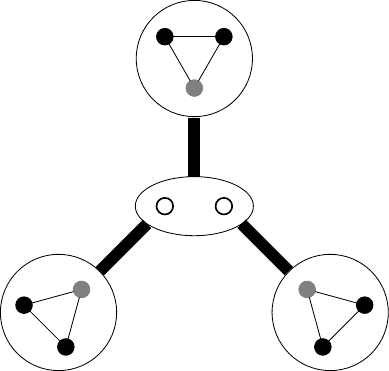}
    \caption{A 3-BG graph with 3 components and a cut set of size 2. The gray vertices are a percolating set.}
  \label{fig:3components}
\end{figure}

We require one more lemma before determining the maximum number of blocks in an $r$-BG graph.  

\begin{lem}
\label{lem:second round lemma}
Let $G$ be an $r$-BG graph with at least $r + 1$ vertices and $A_0$ be a cardinality $r$ percolating set of $G$. If $X$ is a cut set of $G$ with $\card{X}<r$, then at least one vertex of $X$ is adjacent to every vertex of $A_0$.  
\end{lem}

\begin{proof}
Since $\card{G}>r$ we have $V(G) \not\subseteq A_0$.  We further claim that $X \not\subseteq A_0$.  Since $X$ is a cut set, $G - X$ contains at least two components.  In the proof of Lemma \ref{lem:component structure lemma} it is shown that each component \(C\) of $G - X$ that is not completely contained in \(A_0\) contains at least $r - |X|$ vertices of $A_0$. So $X \subseteq A_0$ would imply that $X$ together with such a \(C\) would contain all of $A_0$ (a \(C\) must exist since \(\card{G}>r\)). But since each other component must have at least one vertex of $A_0$, this is a contradiction. So \(X\not\subseteq A_0\). We consider two cases.

Case 1: Every component of $G - X$ is contained in $A_0$. 

In this case, the only vertices which remain to be infected are the vertices of $X$ which are not contained in $A_0$.  Since $A_0$ is a percolating set, these vertices become infected at some point.  Hence, at least one such vertex is infected in the second round, i.e., is adjacent to all vertices of $A_0$.

Case 2: Some component of $G - X$ is not contained in $A_0$.

In this case uninfected vertices occur in $X$ as well as in any component $C$, where $C \not\subseteq A_0$.  No component contains every vertex of $A_0$ and $X$ is not a subset of $A_0$, so after the initial round, the number of infected vertices in $C \cup X$ is less than $r$, for any such $C$.  Vertices of $C - A_0$ can only become infected from vertices in $C$ or $X$, so before any such vertex can become infected, at least one vertex of $X$ must be infected.  Hence, at least one vertex of $X$ is infected in the second round, i.e., is adjacent to all vertices of $A_0$.
\end{proof}

From these two lemmas, we have the following result that generalizes Lemma~\ref{lem:2-BG graph structure}: 

\begin{thm}
\label{thm: max number of blocks}
Let \(r\geq 2\). If $G$ is an $r$-BG graph with at least $r + 1$ vertices, then $G$ contains at most $r$ blocks. Moreover, \(r\) blocks is only achieved by \(G=K_{1,r}\) when $r \geq 3$.
\end{thm}

\begin{proof}
Note that blocks are separated by a cut vertex. Let $X = \{v\}$ be a cut vertex of \(G\). By Lemma \ref{lem:second round lemma}, $v$ is adjacent to every vertex of $A_0$. We claim that there cannot be more than a single cut vertex in an $r$-BG graph.  Suppose for contradiction we have a second cut vertex $u$. Each component of $G - \{u\}$ must contain at least one vertex of $A_0$.  But then, $u$ cannot be a cut vertex because these components are still connected by $v$. Thus \(v\) is the only cut vertex in \(G\) and the number of blocks is exactly the number of components of \(G-X\). By Lemma \ref{lem:component structure lemma}, the largest number of components of $G - X$ in general is $r$.\par 
If \(r\geq 3\) Lemma~\ref{lem:component structure lemma} implies at most one component of \(G-X\) is not contained in \(A_0\) since \(\card{X}<r/2\). So the largest number of blocks only occurs when \(G=K_{1,r}\) and every leaf is contained in \(A_0\). 
\end{proof}

Lemma \ref{lem:component structure lemma} also allows us to analyze the structure of $r$-BG graphs with cut sets of size less than $r$.  As an example, we will examine 3-BG graphs.  We know that $1\leq |X|\leq 2$ and components of $G - X$ can either be contained in $A_0$ or not.  We also know from Lemma \ref{lem:component structure lemma} that components not contained in $A_0$ must contain at least $3-|X|$ vertices of $A_0$.\par

Suppose $|X| = 1$. If every component is contained in $A_0$, then we can have either 2 or 3 components. These possibilities are shown by the leftmost and middle graphs in Figure \ref{fig:cutvertex1} (the gray vertices are the vertices of $A_0$).  It is also possible that one component of $G - X$ is not contained in $A_0$.  Since such a component must contain at least $3 - |X| = 2$ vertices of $A_0$, we can only have one such component and the other must be entirely contained within $A_0$, i.e. a leaf.  One example is shown by the rightmost graph in Figure \ref{fig:cutvertex1}.  

\begin{figure}[h]
    \centering
    \includegraphics[scale=1]{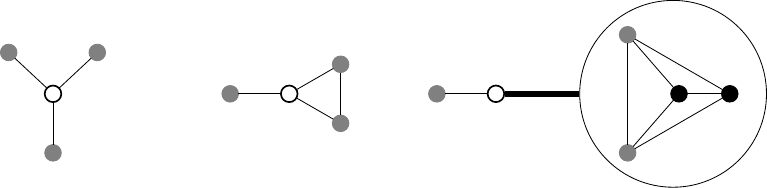}
    \caption{Some cases when $G$ is 3-BG and $G$ has a cut set of size 1. Vertices in the cut set are white, and vertices in $A_0$ are gray.}
\label{fig:cutvertex1}
\end{figure}

Suppose $|X| = 2$. When every component is contained within $A_0$ then we have the same possibilities as before except vertices of $A_0$ now must be adjacent to both vertices of $X$. The leftmost and middle graphs of Figure \ref{fig:cutvertex2} provide examples (it also possible that the vertices of $X$ are adjacent).  If some components of $G - X$ are not contained in $A_0$, then because $|X| = 2$, each such component needs to contain at least one vertex of $A_0$. Hence we may form such a graph by replacing any of the single vertex components with a connected graph of order 2 or more.  The rightmost graph of Figure \ref{fig:cutvertex2} provides such an example, where every vertex of the $K_3$ is joined to $X$. Figure \ref{fig:3components} provides an example where all three components are not subsets of $A_0$. 

\begin{figure}[h]
    \centering
    \includegraphics[scale=1]{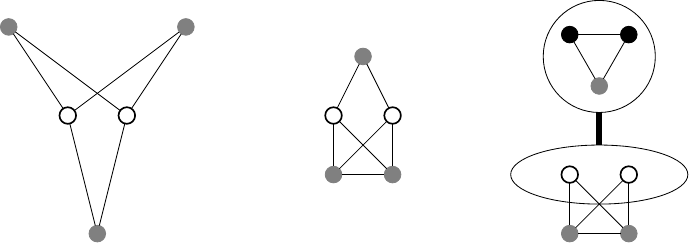}
    \caption{Some cases when $G$ is 3-BG and $G$ has a cut set of size 2. Vertices in the cut set are white, and vertices in $A_0$ are gray.}
\label{fig:cutvertex2}
\end{figure}

In addition to the structure of the components, we have the following result concerning the structure of cut sets in an $r$-BG graph with a cut set of size less than $r$.

\begin{thm}
\label{thm:disjoint cut set}
Let $G$ be an $r$-BG graph with at least $r + 1$ vertices and $A_0$ be a cardinality $r$ percolating set of $G$. If $X$ is a cut set of $G$ with $\card{X}<r$, then there is no cut set $Y$ where $|Y| < r$ and $Y \cap X = \emptyset$.
\end{thm}

\begin{proof}
By Lemma \ref{lem:second round lemma}, there is some vertex $v \in X$ such that $v$ is adjacent to every vertex of $A_0$. Since \(X\cap Y=\emptyset\), it must be that \(X\) is contained within the components of \(G-Y\), thus \(v\) is in a component of \(G-Y\). Each component of \(G-Y\) must contain at least one vertex from \(A_0\) since \(\card{Y}< r\). Since \(v\) is adjacent to all of \(A_0\) we have \(G-Y\) is connected, a contradiction. 
\end{proof}

Another way to extend Lemma \ref{lem:2-BG graph structure} is by generalizing the notion of a block: a \emph{\(k\)-block} of a graph \(G\) is a maximal induced subgraph of \(G\) that is \(k\)-connected. With this notion, a \(2\)-block is just an ordinary block.  Matula \cite{Matula_1978} and Karpov \cite{Karpov_2007} have studied $k$-blocks.  In the case of an ordinary block, we are not concerned with the exact connectivity of the block, only that it is at least 2-connected.  We wish to be more precise with $k$-blocks.  Matula refers to a $k$-block which is not contained in a $k+1$-block as a $k$-ultrablock, while Karpov simply defines a $k$-block as a maximal induced subgraph that is not contained in a $k+1$-block.  In this section, for clarity, we will use Karpov's definition of a $k$-block.

We ask, what is the greatest number of $r$-blocks contained in an $r$-BG graph with a cut set of size less than $r$?  When $r = 2$, this is answered by Lemma \ref{lem:2-BG graph structure}, but for higher $r$, we have the following lower bound: 

\begin{thm}
\label{thm: r block lower bound}
Let $G$ be an $r$-BG graph with at least $r + 1$ vertices. If $X$ is a cut set of $G$ with $\card{X}<r$, then the maximum number of \(r\)-blocks that $G$ contains is at least $r(r-1)$.
\end{thm}

\begin{proof}
We present a construction which contains $r$ components of $G - X$ and where the number of $r$-blocks in each component is $r - 1$.  Let $X$ be a cut set of $G$ containing $r - 1$ vertices.  Furthermore, let $X$ form an independent set of $G$.  Construct each of the \(r\) component of $G - X$ as follows: each component contains a copy of $K_{r-1}$.  Call this an axis.  Join the axis to every vertex of $X$ and also join an independent set of $r-1$ vertices, $S$, to every vertex of the axis.  Join each of the $r - 1$ vertices in \(S\) to a distinct vertex of $X$.  The axis together with each vertex of $S$ and its adjacent vertex of $X$ forms a copy of $K_{r+1}$.  

We now show that $G$ is $r$-BG and that each copy of $K_{r+1}$ is indeed an $r$-block.  Take a distinct vertex from the axis of each component and from these vertices form $A_0$.  Since we have $r$ components, this set of vertices then infects $X$.  The vertices of $X$ together with $A_0$ then infect the other $r - 2$ vertices of each axis.  Lastly, $X$ and the axes of each component infect the remaining vertices.  

Recall that an $r$-block is a maximal $r$-connected subgraph of $G$ which is not contained in an $r+1$-connected subgraph of $G$.  A copy of $K_{r+1}$ is indeed $r$-connected.  Furthermore, if we expand beyond any copy of $K_{r+1}$, the resulting subgraph is no longer $r$-connected.  A $K_{r+1}$ together with an additional vertex of $X$ is disconneced by removing the $r - 1$ vertices of the axis, since $X$ is an independent set and each vertex of $S$ is adjacent only to a single vertex of $X$.  If we expand by including a second vertex of $S$, this is also disconnected by removing the axis.  Lastly, if we include multiple components of $G - X$, these are disconneced by removing the $r - 1$ vertices of $X$.\end{proof}

Figure \ref{fig:3-BG, six 3-blocks example} contains an example of this construction when $r = 3$.  The white vertices are the vertices of $X$, the gray vertices are the vertices of the axes and the black vertices are the remaining vertices of $G$.  

\begin{figure}[h]
    \centering
    \includegraphics[scale=1]{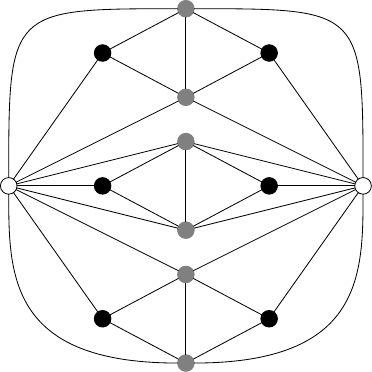}
    \caption{Six 3-blocks in a 3-BG graph}
    \label{fig:3-BG, six 3-blocks example}
\end{figure}

\section{Minimum number of rounds to percolation}

\begin{thm}
Let $G$ be a connected graph with diameter $d$.  Suppose $G$ contains a set of vertices, $A_0$, which percolates with threshold $r$ in $k$ rounds and $|A_0| \leq 2r - 1$.  Furthermore, assume that every vertex in $A_0$ infects some vertex in round 2, i.e., every vertex in $A_0$ is adjacent to at least one vertex in round 2.  Then $k \geq \lceil d/2 \rceil + 1$ and this bound is sharp.  
\end{thm}

\begin{proof}
When numbering the rounds, we refer to the initial round as round 1. Partition $V(G)$ into sets $S_1, S_2, ..., S_k$, where $S_1 = A_0$ and for reach $i$, $S_i$ is the collection of vertices newly infected in round $i$.  Let $p$ be a vertex infected in round $q$, where $q \neq 1$.  Observe that $p$ is adjacent to no more than $r - 1$ vertices in $S_1$ through $S_{q-2}$ (otherwise, $p$ would have become infected in some round from 2 to $q-1$).  Since $p$ is adjacent to at least $r$ vertices in $S_1$ through $S_{q-1}$, then we know that $p$ is adjacent to at least one vertex in $S_{q-1}$.  By iterating this reasoning, we can find a path from a vertex in any round to some other vertex in any previous round. 

Let $u,v$ be two vertices in $G$ where $u \in S_i$ and $v \in S_j$, where $i,j \neq 1$.  If $j \geq i$, then by the above observation, we can form a path $v, v_{j-1}, ..., v_i$, where the index on each vertex is the round in which it was newly infected.  If $v_i = u$, then we have a $u - v$ path.  If not, then we can continue our path starting with $v$ and begin a new path starting with $u$ as follows: $v, v_{j-1}, ..., v_i, v_{i - 1}, ..., v_2$ and $u, u_{i - 1}, ..., u_2$.  If in some round $\ell$ we have $v_\ell = u_\ell$, then we have a $u-v$ path. 

On the other hand, if it is never the case that $v_\ell = u_\ell$, $\ell \geq 2$, we extend the path to the initial round.  Every vertex in $S_2$ must be adjacent to at least $r$ vertices in $S_1$, which implies that in particular, $v_2$ and $u_2$ are each adjacent to at least $r$ vertices in $S_1$. Since $|A_0| = |S_1| \leq 2r - 1$, by the pigeonhole principle, these two sets of $r$ vertices cannot be disjoint.  Hence, we can choose some $v_1 = u_1$ and we form a $u-v$ path.  A diagram of this process is shown in Figure \ref{fig:uvpath}.

Now, suppose that both $u$ and $v$ are in $S_1$.  Since we assumed that every vertex in $S_1$ infects at least one vertex in $S_2$, both $u$ and $v$ are adjacent to a vertex in $S_2$.  If both are adjacent to the same vertex, the we have a $u-v$ path of length 2.  If $u$ and $v$ are not adjacent to the same vertex, then we have two paths $u,u_2$ and $v,v_2$.  Since $v_2$ and $u_2$ are each adjacent to $r$ vertices in $S_1$, by the pigeonhole principle, $v_2, u_2$ are mutually adjacent to some $w \in S_1$ and so we have $v,v_2,w,u_2,u$, a $u-v$ path of length 4.  If only $u$ is in $S_1$, then by similar reasoning, we either have a $u-v$ path of length $j-1$ or a $u-v$ path of length $j+1$.  

Since we can use this method to construct a path between any two vertices in $G$, the diameter of $G$ cannot be any longer than the longest possible such path.  This occurs when both $u$ and $v$ are infected in the final round.  Since it takes $k - 1$ steps to go from the $k^{th}$ round to the $1^{st}$ round, we can write $d \leq 2k - 2$.  Solving for $k$ yields $d/2 + 1 \leq k$ and since the number of rounds must be an integer, we have $\lceil d/2 \rceil + 1 \leq k$.

Without the additional assumption that every vertex in $S_1$ infects some vertex in $S_2$, it is possible that at most $r - 1$ vertices in $S_1$ are adjacent to no vertices in $S_2$.  In which case, a path from a vertex in $S_k$ to a vertex in $S_1$ can have length at most $k + r - 2$ and then our lower bound depends on both $r$ and $d$ rather than $d$ alone.  
\end{proof}

\begin{figure}[h]
    \centering
    \includegraphics[scale=1]{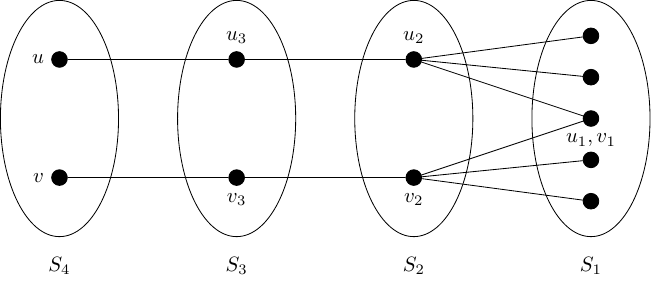}
    \caption{Finding a $u-v$ path.}
    \label{fig:uvpath}
\end{figure}

This bound is sharp.  Consider the following class of graphs.  Begin with $P_n$ and replace every vertex by a set of $r$ independent vertices.  Label these sets $B_1, ..., B_n$, where $B_i$ corresponds with vertex $i$ of $P_n$, and vertices labelled from left to right.  Join every vertex in $B_1$ to every vertex of $B_2$, and in general, every vertex in $B_i$ to every vertex in adjacent sets.  We denote a member of this family of graphs by $P_{n,r}$.  Figure \ref{fig:p52} shows this construction for $P_{5,2}$ and $P_{6,2}$. 

The diameter of graphs in this family is $n - 1$.  If we initially infect the middle set of vertices (for a graph where $n$ is odd), then the entire graph is infected when $B_1$ and $B_n$ become infected, which occurs after $\frac{n-1}{2} + 1$ rounds.  If $n$ is even, then if we initially infect either of the two centermost sets, the infection percolates when either $B_1$ or $B_n$ becomes infected (whichever is furthest from our starting set).  This requires $\frac{n}{2} + 1 = \lceil \frac{n-1}{2} \rceil + 1$ rounds.  In either case, we can see that the lower bound of $\lceil d/2 \rceil + 1$ rounds is attained.  

\begin{figure}[h]
    \centering
    \includegraphics[scale=1]{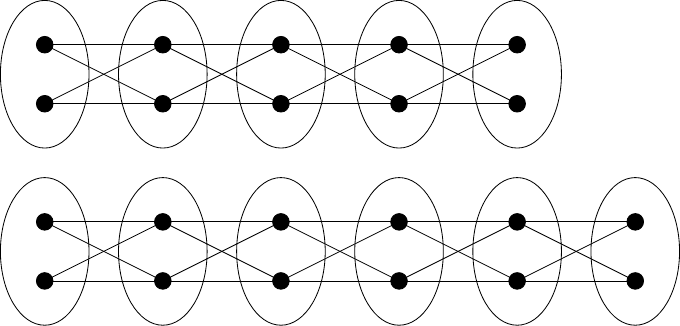}
    \caption{$P_{5,2}$ and $P_{6,2}$}
    \label{fig:p52}
\end{figure}

\begin{thm}
Let $G$ be a connected graph with a set of vertices $A_0$, which percolates in $k$ rounds with percolation threshold $r$.  If $|A_0| = r$, then $k \geq rad(G) + 1$ and this bound is sharp.  
\end{thm}

\begin{proof}
Let $x$ be a vertex in $S_1$ and $y$ be a vertex in $S_i$, $1 < i \leq k$.  Using the same method as in Theorem 3.1, we can form a path $y,y_{i-1}, ..., y_{2}$, where $y_j \in S_j$.  Since $A_0$ contains exactly $r$ vertices, every vertex in $S_2$ must be adjacent to every vertex in $A_0$.  Hence, $y_2$ is adjacent to  $x$ and $y,y_{i-1}, ..., y_{2}, x$ is an $x-y$ path of length $ i - 1$.  If $y \in S_1$, then we can construct $y,y_2,x$, an $x-y$ path of length 2.

The greatest length of such a path is $k - 1$.  Since we can form a path from every vertex in $G$ to $x \in S_1$, we know that the eccentricity of $x$, $e(x)$, is at most $k - 1$.  We then have the following inequality: $rad(G) \leq e(x) \leq k - 1$.  Hence, $k \geq rad(G) + 1$.  This inequality is sharp because each $P_{n,r}$ contains a set of vertices which percolates in $rad(G) + 1$ rounds.
\end{proof}

\section{Maximum number of rounds to percolation}

In this section, we construct a family of graphs which show that given percolation threshold $r$ and diameter $d$, the number of rounds before the infection percolates is not bounded above.  We first construct a family of graphs with diameter 2 and with threshold $r = 2$ and then generalize the construction for arbitrary diameter and percolation threshold. 

We begin constructing $G$ by selecting an independent set of vertices.  Call this set $A_0$.  This set must contain at least $r$ vertices, but other than this there is no restriction on the cardinality of this set.  Next, join every vertex of $A_0$ to a vertex $x_1$.  After this, construct a path of length $s$ and denote the vertices $y_1, y_2, ..., y_s$ from left to right.  Join every vertex of the path to $x_1$ and join $y_1$ to exactly one vertex in $A_0$.  An example of this construction for $r = 2, s = 5$ is shown in Figure \ref{fig:d2r2}.

If we select $A_0$ as our initial set of infected vertices, then the infection percolates in $s + 2$ rounds.  This is because each vertex of the path cannot become infected until the previous vertex of the path is infected and $y_1$ cannot become infected until after $x_1$ is infected.  Since $x_1$ is a dominating vertex, our graph is diameter 2.

\begin{figure}[h]
    \centering
    \includegraphics[scale=1]{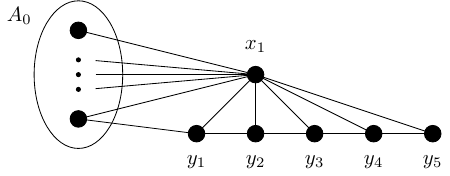}    
    \caption{A diameter 2 graph which percolates in 7 rounds with percolation threshold 2}
    \label{fig:d2r2}
\end{figure}

We generalize the construction as follows.  First, we construct $P_{(d-1),r}$.  We then join every vertex of the the $d-1^{st}$ set (the last one on the right) to $x_1, x_2, ..., x_{r-1}$.  Lastly, we form a path on $s$ vertices $y_1, y_2, ..., y_s$ and join every vertex of the path to $x_1, x_2, ..., x_{r-1}$.  Next, we join $x_1$ to a single vertex in the $d-1^{st}$ set.  The set of vertices $\{x_1, ..., x_{r-1}\}$ ensures that every vertex of the path $y_1, ..., y_s$ is within distance $d$ of our other vertices.  An example of this construction with diameter 4 and percolation threshold 3 is shown in Figure \ref{fig:d4r3}.

If we select the leftmost set of $r$ vertices of $P_{(d-1),r}$ as our initial set, such a graph percolates in $d + s$ rounds.  First the infection percolates through the $d-1$ sets.  After this, $x_1, ..., x_{r-1}$ become infected.  Next, $y_1$ becomes infected and then each $y_i$ becomes infected in turn.

\begin{figure}[h]
    \centering
    \includegraphics[scale=1]{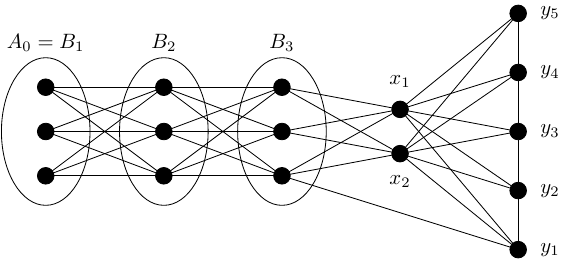}  
    \caption{A diameter 4 graph which percolates in 9 rounds with percolation threshold 3}
    \label{fig:d4r3}
      
\end{figure}

Although diameter is insufficient for an upper bound on the number of rounds, an upper bound using a different graph invariant is possible.  The \textit{detour distance} between two vertices $u,v$, denoted $D(u,v)$ is the length of the longest path between $u$ and $v$.  The \textit{detour eccentricity}, of a vertex, $v$, denoted $e_D(v)$, is the longest detour distance from $v$ to any other vertex.  The \textit{detour diameter} of a graph, $G$, denoted $diam_D(G)$ is the largest detour eccentricity among vertices in $G$.  Observe that the detour diameter is the length of the longest path in $G$.  These definitions and other facts about the detour distance are stated in Chartrand et. al. \cite{Chartrand_2004}.

\begin{thm}
If $G$ is a connected graph containing a percolates set which $r$-percolates in $k$ rounds, then  $k \leq diam_D(G)$ + 1.  
\end{thm}

\begin{proof}
Using the same process as in in the proof of Theorem 3.1, we form a path from a vertex in the $k^{th}$ round to a vertex in the initial round, where each vertex of the path is in a different round.  Such a path has length $k-1$.  Since $diam_D(G)$ is the length of the longest path in $G$, we know that $k - 1 \leq diam_D(G)$.  Hence, $k \leq diam_D(G) + 1$. 
\end{proof}

The above theorem gives an upper bound of $diam_D(G) + 1$, based on the idea that a path from the last round to the first round is in fact the longest path in the graph.  But when $r \geq 2$, we were able to extend such a path by at least one vertex in all the examples we examined.  Hence, we conjecture that when $r \geq 2$ the upper bound is in fact, $diam_D(G)$.

\begin{conj}
If $G$ is a connected graph containing a percolating set which $r$-percolates in $k$ rounds and $r \geq 2$, then $k \leq diam_D(G)$.  
\end{conj}

For every $r$ and every value of $diam_D(G)$, it is possible to find a graph which percolates in $diam_D(G)$ rounds with threshold $r$.  This family of graphs are all caterpillars.  A \textit{caterpillar} is a tree which consists of a central path, each vertex of which has some number of leaves (possibly 0).  We construct these graphs as follows.  Form a caterpillar with a central path of length $diam_D(G) - 2$ and where all vertices of the path except the leftmost endpoint have $r - 1$ leaves.  The leftmost endpoint has $r$ leaves.  The longest path in such a graph is formed by moving from a leaf of the leftmost endpoint along the central path to a leaf of the rightmost endpoint.  If we begin the percolation process by infecting the leaves of the path, then such a graph percolates in $diam_D(G)$ rounds.  Figure \ref{fig:caterpillar} shows an example of such a graph for $r = 3$ and detour diameter $7$.  

\begin{figure}[h]
    \centering
    \includegraphics[scale=1]{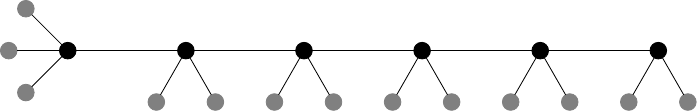}
    \caption{A graph which 3-percolates in $diam_D(G)$ rounds when $A_0$ is the set of gray vertices}
    \label{fig:caterpillar}
\end{figure}

\section{Open Questions}

1. The two lower bounds on the number of rounds to percolation given in this paper are based on radius and diameter.  It would be interesting to see lower bounds based on other graph invariants, or bounds for specific graph classes.  

2. Likewise the upper bound given in this paper is unconditional, but it is likely that the actual largest number of rounds for most graphs is substantially smaller than the length of the longest path.  Given other assumptions, what is the maximum number of rounds to percolation?  

3. What is the largest number of $r$-blocks contained in an $r$-BG graph with a cut set of size less than $r$? Other results on the structure of cut sets of an $r$-BG graph would also be interesting.  

\section{Acknowledgements}
The authors wish to thank Craig Larson for discussions which initiated and improved this paper and Ghidewon Abay-Asmerom for suggesting the concept of detour distance. RI was partially supported by NSF DMS--2204148 and by The Thomas F. and Kate Miller Jeffress Memorial Trust, Bank of America, Trustee.

\printbibliography

\end{document}